\documentclass[preprint,10pt]{article}

\usepackage{amssymb}
\usepackage{amsthm}
\usepackage{amsmath}
\usepackage{color}
\usepackage{subfigure}
\usepackage{comment}
\usepackage{authblk}
\usepackage{graphicx}

\theoremstyle{plain}
\newtheorem{thm}{Theorem}[section]
\newtheorem{lem}[thm]{Lemma}
\newtheorem{prop}[thm]{Proposition}

\begin{document}

\title{Homomorphism-homogeneous $L$-colored graphs}

\author[1]{David Hartman}
\author[1]{Jan Hubi\v{c}ka}
\author[2]{Dragan Ma\v{s}ulovi\'{c}}

\affil[1]{Computer Science Institute of Charles University in Prague, Malostransk\' e n\' am\v est\' i 25, 118~00, Prague 1, Czech Republic}
\affil[2]{Department of Mathematics and Informatics, Faculty of Science, University of Novi Sad, Trg D. Obradovi\'{c}a 4, 21000 Novi Sad, Serbia}

\date{}

\maketitle

\begin{abstract}
A relational structure is homomorphism-homogeneous ($HH$-homogeneous for short) if every homomorphism between finite induced substructures of the structure can be extended to a homomorphism over the whole domain of the structure. Similarly, a structure is monomorphism-homogeneous ($MH$-homogeneous for short) if every monomorphism between finite induced substructures of the structure can be extended to a homomorphism over the whole domain of the structure. In this paper we consider $L$-colored graphs, that is, undirected graphs without loops where sets of colors selected from $L$ are assigned to vertices and edges. A full classification of finite $MH$-homogeneous $L$-colored graphs where $L$ is a chain is provided, and we show that the classes $MH$ and $HH$ coincide. When $L$ is a diamond, that is, a set of pairwise incomparable elements enriched with a greatest and a least element, the situation turns out to be much more involved. We show that in the general case the classes $MH$ and $HH$ do not coincide.
\end{abstract}

\section{Introduction}\label{intro}

A {\em relational structure} $\mathcal{A}$ is a pair $(A,R_\mathcal{A})$, where $R_\mathcal{A}$ is a tuple $(R_\mathcal{A}^i:i\in I)$ of relations 
such that $R_\mathcal{A}^i\subseteq A^{\delta_i}$ (i.e. $R_\mathcal{A}^i$ is a $\delta_i$-ary
relation on $A$). The family $\Delta = (\delta_i: i\in I)$ is called the {\em
type} of $\mathcal{A}$. The type is usually fixed and understood from the context.
The underlying set $A$ is called the {\em domain} of $\mathcal{A}$.

Relational structures of type $(2)$ can be seen as directed graphs with loops.
We will also consider undirected graphs without loops as relational
structures of type $(2)$ with one symmetric and irreflexive binary relation.

For structures $\mathcal{A} = (A,R_{\mathcal{A}})$ and $\mathcal{B} = (B,R_{\mathcal{B}})$ a {\em homomorphism} $f : \mathcal{A} \rightarrow \mathcal{B}$ is a mapping $f : A \rightarrow B$ such that $(x_1,x_2, \ldots, x_{\delta_i}) \in R_{\mathcal{A}}^i$ implies $(f(x_1),$$f(x_2), \ldots, f(x_{\delta_i})) \in R_{\mathcal{B}}^i$ for each $i \in I$.  If $f$ is one-to-one then $f$ is called a {\em monomorphism}.  An {\em isomorphism} $g : \mathcal{A} \rightarrow \mathcal{B}$ is a bijective mapping $g : A \rightarrow B$ such that $(x_1,x_2, \ldots, x_{\delta_i}) \in R_{\mathcal{A}}^i \Leftrightarrow (g(x_1),$$g(x_2), \ldots, g(x_{\delta_i})) \in R_{\mathcal{B}}^i$ for each $i \in I$. 

An isomorphism from a structure to itself is called an {\em automorphism}. Similarly, an {\em endomorphism} is a homomorphism from a structure to itself. Throughout the paper a we write
\begin{equation*}
\displaystyle f = \begin{pmatrix}
  x_1 & x_2 & \ldots & x_n\\
  y_1 & y_2 & \ldots & y_n \end{pmatrix}
\end{equation*}
for a mapping $f : \{x_1,x_2, \ldots, x_n\} \rightarrow \{y_1,y_2, \ldots, y_n\}$ such that $f(x_i) = y_i$ for all $i \in \{1,2, \ldots, n\}$.

A structure $\mathcal{A}$ is called {\em ultrahomogeneous} if every isomorphism between two induced finite substructures of $\mathcal{A}$ can be extended to an automorphism of $\mathcal{A}$. There is a long-standing effort to classify all ultrahomogeneous relational structures since the work of Fraiss\'{e}~\cite{Fraisse:1953} (see, for example \cite{Cherlin:1998, Schmerl:1979}). 

In this paper we will use the classification of finite undirected graphs without
loops provided by Gardiner in \cite{Gardiner:1976}. He has shown that a finite
graph is ultrahomogeneous if and only if it is isomorphic to one of the following graphs:
\begin{enumerate}
\item[1.] a disjoint union of complete graphs all of the same size, $\bigcup_{i=1}^k K_n$,
\item[2.] multipartite graphs $K_{n_1,n_2, \ldots, n_k}$ with $n_i=n_j= \ldots = n_k$,
\item[3.] the $5$-cycle $C_5$,
\item[4.] the line graph $L(K_{3,3})$.
\end{enumerate}

While this class of finite ultrahomogeneous graphs is important in our case we will refer to it as the {\em Gardiner's class} or simply as {\em Gardiner graphs}.% and denote it by $\mathcal{G}$.

%\section{Morphism extension classes}

Quite recently, Cameron and Ne\v{s}et\v{r}il introduced the following variant of homogeneity~\cite{CameronNesetril:2006}. A structure $\mathcal{A}$ is called {\em homomorphism-homogeneous} ({\em $HH$-homogeneous} for short) if every homomorphism between finite induced substructures of $\mathcal{A}$ can be extended to an endomorphism of $\mathcal{A}$.  This notion has motivated a new classification programme.
Finite $HH$-homogeneous undirected graphs are classified as complete and null graphs~\cite{CameronNesetril:2006}. Other classes of structures where $HH$-homogeneous structures have been fully classified are, for example, partially ordered sets in~\cite{Masulovic:2007} or in~\cite{CameronLockett:2010} and finite tournaments~\cite{IlicEtAl:2008}. 

\begin{figure}[ht]
\begin{center}
\includegraphics[scale=1.0]{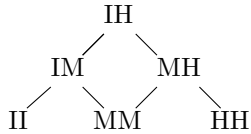}
\caption{The hierarchy of morphism extension classes for a general relational structure.}
\label{fig:morphism_clasess_categories}
\end{center}
\end{figure}

Several other variants of homogeneity are also proposed in~\cite{CameronNesetril:2006}.  For these we follow the notation used in~\cite{CameronNesetril:2006, CameronLockett:2010}. We say that a structure $\mathcal{A}$ belongs to a class $XY$ if every $x$-morphism from a finite substructure of $\mathcal{A}$ into $\mathcal{A}$ extends to a $y$-morphism from $\mathcal{A}$ to $\mathcal{A}$ where pairs $(X,x)$ and $(Y,y)$ can be $(I,iso)$, $(M,mono)$ and $(H,homo)$. 

Many of these classes are related. For example $MH$ is a subclass of $IH$.
The obvious inclusions between the morphism extension classes are depicted in Figure~\ref{fig:morphism_clasess_categories}.
Note that, for simplicity, we omit the inclusions implied by transitivity in
all diagrams.

\begin{figure}[ht]
\begin{center}
\subfigure[Countably infinite graphs]{\includegraphics[scale=1.0]{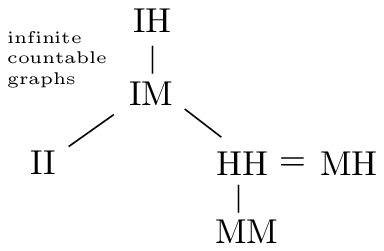}} \qquad \qquad
\subfigure[Finite graphs]{\includegraphics[scale=1.0]{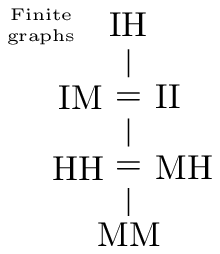}}
\caption{The hierarchy of morphism extension classes for graphs~\cite{RusinovSchweitzer:2010}.}
\label{fig:mclasses_graphs}
\end{center}
\end{figure}

For specific types of relational structures, some classes are known to be equivalent (such as $HH$ and $MH$ for
graphs~\cite{CameronNesetril:2006,RusinovSchweitzer:2010}). This leads to simplified inclusion diagrams.
Figure~\ref{fig:mclasses_graphs} depicts the hierarchy for finite and infinite countable graphs~\cite{RusinovSchweitzer:2010}, and Figure~\ref{fig:mclasses_posets} the hierarchy for partially ordered sets~\cite{CameronLockett:2010}.

\begin{figure}[ht]
\begin{center}
\subfigure[Countably infinite posets]{\includegraphics[scale=1.0]{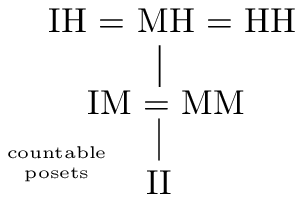}} \qquad \qquad
\subfigure[Finite posets]{\includegraphics[scale=1.0]{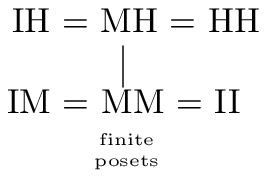}}
\caption{Hierarchy of morphisms extension classes for partially ordered sets~\cite{CameronLockett:2010}.}
\label{fig:mclasses_posets}
\end{center}
\end{figure}

The main question of the classification programme is to give a catalogue of structures belonging to a given class. The full classification of any of the classes is far from complete. The class $II$ is the most extensively studied one, while the class $HH$ and other variants are less explored. In Section~\ref{s:multicolored_graphs} we introduce a rather general notion of $L$-colored graphs where $L$ is a partially ordered set. (We think of $L$ as a poset of admissible combinations of colors ordered by inclusion.) In Sections \ref{s:chains} and \ref{s:diamonds} we provide classifications of finite $MH$-homogeneous $L$-colored graphs where $L$ is a chain or a diamond. In all the existing classification results, the classes $HH$ and $MH$ coincide. This leads to the question whether there is a structure that is $MH$ but not $HH$.  We give a positive answer to this question in Section~\ref{s:diamonds}. A few more types of structures where the classes $MH$ and $HH$ do not coincide are given in Section~\ref{s:mor-ext-classes}.

\section{Multicolored graphs}
\label{s:multicolored_graphs}

Let $G = (V,\mathbf{E})$ with $\mathbf{E} = (E_1,E_2,\ldots,E_m)$ be a relational structure with a collection $\mathbf{E}$ of symmetric irreflexive binary relations. This structure is called a {\em multicolored graph}. In case $m = 2$ we say that $G$ is a {\em bicolored graph}, or shortly a {\em bigraph}. Finite $HH$-homogeneous bigraphs have been classified in~\cite{HartmanMasulovic:2011}.%\marginpar{\textsf{DAVID, ADD A REFERENCE HERE}}

In this paper we propose the study of a related but more general notion which yields a clearer, unifying presentation.
Let $L$ be a partially ordered set with the ordering relation $\preceq$, with the least element~$0$ and the greatest element~$1$. An {\em $L$-colored graph} is an ordered triple $(V, \chi', \chi'')$ such that $V$ is a nonempty set, $\chi' : V \to L$ is an arbitrary function and $\chi'' : V^2 \to L$ is a function satisfying the following:
\begin{enumerate}
\item $\chi''(x,x) = 0$; and
\item $\chi''(x,y) = \chi''(y,x)$ whenever $x \neq y$.
\end{enumerate}
The function $\chi'$ provides colors of vertices of $G$, while $\chi''$ provides colors of edges of $G$. The two restrictions that we have imposed on $\chi''$ mean that $G$ is without loops and undirected.

A multicolored graph $(V, (E_1, \ldots, E_m))$ as introduced in \cite{HartmanMasulovic:2011} can be thought of as an $L$-colored graph $(V, \chi', \chi'')$ where $L =  \mathcal{P}(\{1,2,\ldots,m\})$ with set-inclusion as the ordering relation, $\chi'(x) = \emptyset$ (no colors are assigned to vertices), and $\chi''(x,y) = \{j:\{x,y\} \in E_j\}$.

Consequently, the intuition that we have is that $\chi'(x) = 0$ means that there are no colors assigned to $x$, and $\chi'(x) = 1$ means that the vertex $x$ is colored by all the available colors. Analogously, $\chi''(x, y) = 0$ means that $x$ and $y$ are nonadjacent, while $\chi''(x, y) = 1$ means that the edge $\{x, y\}$ is colored by all the available colors.

A {\em homomorphism} between two $L$-colored graphs $(V_1,\chi'_1,\chi''_1)$ and $(V_2,\chi'_2,\chi''_2)$ is a mapping $f:V_1 \to V_2$ such that
\begin{equation*}
\chi'_1(x) \preceq \chi'_2(f(x)) \text{ and } \chi''_1(x,y) \preceq \chi''_2(f(x),f(y)),
\end{equation*}
for all $x$ and $y$ in $V_1$.

For $W \subseteq V$, a \emph{substructure of $(V,\chi',\chi'')$ induced by $W$} is $G[W] = (W,\chi'|_W, \chi''|_W)$, where $\chi'|_W$ and $\chi''|_W$ denote the restrictions of $\chi'$ and $\chi''$ to $W$, respectively.

For an $L$-colored graph $G = (V, \chi', \chi'')$ and $\alpha \in L$ let $W_\alpha = \{x \in V : \chi'(x) = \alpha \}$ and $G^{(\alpha)} = G[W_\alpha]$.

We say that an $L$-colored graph $G = (V,\chi',\chi'')$ is {\em homomorphism-homogeneous} ({\em $HH$-homogeneous} for short) if every homomorphism $f:S \rightarrow T$ between finite induced substructures of $G$ extends to an endomorphism of~$G$. We say that an $L$-colored graph $G = (V,\chi',\chi'')$ is {\em $MH$-homogeneous} if every monomorphism $f:S \rightarrow T$ between finite induced substructures of $G$ extends to an endomorphism of~$G$.

Let $G = (V, \chi', \chi'')$ be an $L$-colored graph, and let $\theta_G \subseteq V^2$ be the reflexive transitive closure of $\theta^0_G = \{(x, y) \in V^2 : \chi''(x, y) \ne 0\}$. Then $\theta_G$ is an equivalence relation on $V$ whose equivalence classes will be referred to as {\em connected components} of $G$. An $L$-colored graph $G$ is \emph{connected} if $\theta_G$ has only one equivalence class. Otherwise, it is \emph{disconnected}. We say that $G$ is \emph{complete} if $\chi''(x, y) \ne 0$ for all $x \ne y$.

An $L$-colored graph $G = (V, \chi', \chi'')$ is \emph{vertex-uniform} if there exists an $\alpha \in L$ such that $\chi'(x) = \alpha$ for all vertices $x$, and it is \emph{edge-uniform} if there exists a $\beta \in L \setminus \{0\}$ such that $\chi''(x, y) = \beta$ for all vertices $x$, $y$ such that $x \ne y$. We say that an $L$-colored graph $G = (V, \chi', \chi'')$ is \emph{uniform} if it is both vertex-uniform and edge-uniform. Up to isomorphim, a finite connected uniform graph is uniquely determined by $n=|V|$, the color of vertices $\alpha$ and the color of edges $\beta \succ 0$, and we denote it by~$U(n, \alpha, \beta)$.

If there is no danger of confusion, we shall write simply $\chi(x)$ and $\chi(x, y)$ instead of $\chi'(x)$ and $\chi''(x, y)$, respectively. Also, the set of vertices of $G$ will be denoted by $V(G)$.

\begin{lem}\label{lem:pump-pump}
Let $G$ be an $MH$-homogeneous $L$-colored graph. Assume that there exist three distinct vertices
$a_0, a_1, x \in V(G)$ such that:
\begin{itemize}
\item[(i)] $\chi(a_0, a_1) \succ 0$ and $\chi(x, a_1) \succ 0$,
\item[(ii)] $\chi(a_0, x) \preceq \chi(a_0, a_1)$ and $\chi(x) \preceq \chi(a_1)$, and
\item[(iii)] $\chi(a_0, x) \prec \chi(a_0, a_1)$ or $\chi(x) \prec \chi(a_1)$.
\end{itemize}
Then $G$ is not finite.
\end{lem}
\begin{proof}
Let us construct inductively a sequence of mappings $f_2$, $f_3$, \ldots, and a sequence of vertices $a_2, a_3, \ldots \in V(G)$ with the following properties:
\begin{itemize}
\item[(1)]
   Let $m(n) = \max\{j \in \{1, \ldots, n\} : \chi(x, a_j) \succ 0\}$. (Note that $m(n) \ge 1$ due to~(i)). The mapping
   $\displaystyle f_{n+1} = \begin{pmatrix}
      a_0 & \ldots & a_{m(n) - 1} & a_{m(n) + 1} & \ldots & a_n & x\\
      a_0 & \ldots & a_{m(n) - 1} & a_{m(n) + 1} & \ldots & a_n & a_{m(n)} \end{pmatrix}$
    is a monomorphism from $G[a_0, \ldots, a_{m(n) - 1}, x,  a_{m(n) + 1}, \ldots, a_n]$ to $G[a_0, \ldots, a_n]$.
\item[(2)]
    $G$ is $MH$-homogeneous so there is an endomorphism $f_{n+1}^*$ of $G$ which extends~$f_{n+1}$ and we let
    $a_{n+1} = f_{n+1}^*(a_{m(n)})$.
\item[(3)]
    $a_{n+1} \notin \{x, a_0, \ldots, a_n\}$.
\item[(4)]
    $\chi(a_i, a_j) \succ 0$ for all $i, j \in \{0, \ldots, n+1\}$ such that $i \ne j$.
\item[(5)]
    $\chi(a_0, a_1) \preceq \chi(a_0, a_j)$ and $\chi(a_1) \preceq \chi(a_j)$, for all $1 \le j \le n+1$.
\item[(6)]
    $\chi(a_j, x) \preceq \chi(a_j, a_k)$ and $\chi(x) \preceq \chi(a_k)$ for all $0 \le j < k \le n+1$.
\item[(7)]
    $\chi(a_0, x) \prec \chi(a_0, a_j)$ or $\chi(x) \prec \chi(a_j)$, for all $1 \le j \le n+1$. 
\end{itemize}

The inductive construction proceeds in several steps, and the corresponding $L$-colored subgraphs can be depicted as in Figure~\ref{fig:main_lemma_induction}.

\begin{figure}[ht]
\begin{center}
\includegraphics[scale=1.0]{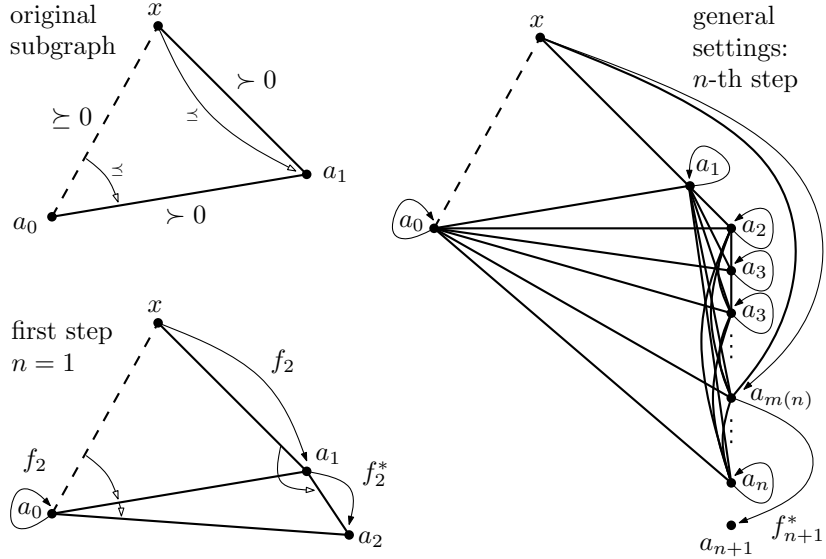}
\caption{The original subgraph, the first step and general settings for the inductive construction. Bold lines without arrows represent edges---solid lines are those having colors $\succ 0$ and dashed those having colors $\succeq 0$. Thin lines with full arrows represents mappings and thin lines with empty arrows indicate the direction of the succession in colors.}
\label{fig:main_lemma_induction}
\end{center}
\end{figure}

The mapping $\displaystyle f_2 = \begin{pmatrix} a_0 & x \\ a_0 & a_1 \end{pmatrix}$ is a monomorphism from $G[a_0, x]$ to $G[a_0, a_1]$ by (ii), while in case $n > 2$ the requirement $(6)$ for $n$ (inductive hypothesis) and the fact that $\chi(a_j, x) = 0$ for $j > m(n)$ ensure that $f_{n+1}$ is a monomorphism from from $G[a_0, \ldots, a_{m(n) - 1}, x,  a_{m(n) + 1}, \ldots, a_n]$ to $G[a_0, \ldots, a_n]$. This shows~$(1)$.

Let us show that $(3)$ holds for $a_{n+1}$ assuming $(1)$--$(7)$ for $n$.
\begin{itemize}
\item
if $a_{n+1} = x$ then $\chi(a_0, a_{m(n)}) \preceq \chi(f_{n+1}^*(a_0), f_{n+1}^*(a_{m(n)})) = \chi(a_0, x)$ and $\chi(a_{m(n)}) \preceq \chi(f_{n+1}^*(a_{m(n)})) = \chi(x)$, which contradicts~(7);
\item
if $a_{n+1} = a_{m(n)}$ then $0 = \chi(a_{m(n)}, a_{n+1}) = \chi(f_{n+1}^*(x), f_{n+1}^*(a_{m(n)})) \succeq \chi(x, a_{m(n)})$, but  $\chi(x, a_{m(n)}) \succ 0$ by definition of $m(n)$ -- contradiction;
\item
if $a_{n+1} = a_j$ for some $j \ne m(n)$ then, by (4), $0 \prec \chi(a_j, a_{m(n)}) \preceq \chi(f_{n+1}^*(a_j), f_{n+1}^*(a_{m(n)})) = \chi(a_j, a_{n+1}) = \chi(a_j, a_j) = 0$ -- contradiction.
\end{itemize}

Let us show that $(4)$ holds for $a_{n+1}$ assuming $(1)$--$(7)$ for $n$. Clearly, it suffices to show that $\chi(a_i, a_{n+1}) \succ 0$ for all $0 \le i \le n$.
\begin{itemize}
\item
if $i \ne m(n)$ then $\chi(a_i, a_{n+1}) = \chi(f_{n+1}^*(a_i), f_{n+1}^*(a_{m(n)})) \succeq \chi(a_i, a_{m(n)}) \succ 0$ by the induction hypothesis;
\item
if $i = m(n)$ then $\chi(a_{m(n)}, a_{n+1}) = \chi(f_{n+1}^*(x), f_{n+1}^*(a_{m(n)})) \succeq \chi(x, a_{m(n)}) \succ 0$ by definition of $m(n)$.
\end{itemize}

To see that (5) holds for $a_{n+1}$ we use the induction hypothesis and the fact that $f_{n+1}^*$ is a homomorphism:
\begin{itemize}
\item
$\chi(a_0, a_1) \preceq \chi(a_0, a_{m(n)}) \preceq \chi(f_{n+1}^*(a_0), f_{n+1}^*(a_{m(n)})) = \chi(a_0, a_{n+1})$;
\item
$\chi(a_1) \preceq \chi(a_{m(n)}) \preceq \chi(f_{n+1}^*(a_{m(n)})) = \chi(a_{n+1})$.
\end{itemize}

Let us show that (6) holds for $a_{n+1}$. As above, from (ii) and (5) we immediately get $\chi(x) \preceq \chi(a_1) \preceq \chi(a_{n+1})$. To see that $\chi(a_j, x) \preceq \chi(a_j, a_{n+1})$ for all $j \in \{0, \ldots, n\}$ we consider several cases: 
\begin{itemize}
\item
if $j > m(n)$ then $\chi(a_j, x) = 0$ by definition of $m(n)$ so  $\chi(a_j, x) \preceq \chi(a_j, a_{n+1})$ holds trivially;
\item
if $j < m(n)$ then using the induction hypothesis and the fact that $f_{n+1}^*$ is a homomorphism we get $\chi(a_j, x) \preceq \chi(a_j, a_{m(n)}) \preceq \chi(f_{n+1}^*(a_j), f_{n+1}^*(a_{m(n)})) = \chi(a_j, a_{n+1})$;
\item
if $j = m(n)$ then $\chi(a_{m(n)}, x) \preceq \chi(f_{n+1}^*(a_{m(n)}), f_{n+1}^*(x)) = \chi(a_{n+1}, a_{m(n)})$.
\end{itemize}

Finally, (7) follows from (5) and (iii).

Therefore, $G$ contains an infinite sequence $a_0, a_1, a_2, \ldots$ of pairwise distinct vertices, so it cannot be finite.
\end{proof}

In the rest of the paper we restrict our attention to two types of partially ordered sets~$L$: chains and diamonds.

\section{$L$-colored graphs over chains}
\label{s:chains}

In this section we classify finite $MH$-homogeneous $L$-colored graphs where $L$ is a bounded chain and show that in this setting the classes $MH$ and $HH$ coincide. So, let $L$ be a chain with the least element~0 and the greatest element~1.

\begin{lem}\label{lem.1}
Let $G$ be a finite $L$-colored graph which is $MH$-homogeneous. Assume that $x, y, z$ are three distinct vertices of $G$ satisfying $\chi(x, z) \succ 0$ and $\chi(y, z) \succ 0$. Then:

  $(a)$ $\chi(x, y) \prec \chi(x, z)$ if and only if $\chi(y) \succ \chi(z)$;

  $(b)$ $\chi(x, y) = \chi(x, z)$ if and only if $\chi(y) = \chi(z)$.
\end{lem}
\begin{proof}
Clearly, $(b)$ follows immediately from $(a)$ because $L$ is a chain. Let us show $(a)$. Suppose that $G$ is a finite $MH$-homogeneous $L$-colored graph, and let $x, y, z$ be three distinct vertices of $G$ satisfying $\chi(x, z) \succ 0$ and $\chi(y, z) \succ 0$ but not $(a)$. Then either
  $$
    \chi(x, y) \prec \chi(x, z) \text{\quad and\quad} \chi(y) \preceq \chi(z)
  $$
or
  $$
    \chi(x, y) \succeq \chi(x, z) \text{\quad and\quad} \chi(y) \succ \chi(z).
  $$
In both cases finiteness of $G$ contradicts Lemma~\ref{lem:pump-pump}.
\end{proof}

\begin{lem}\label{lem.2}
Let $G$ be a finite $MH$-homogeneous $L$-colored graph. Then:

  $(a)$ for every $\alpha \in L$, every connected component of $G^{(\alpha)}$ is a uniform graph;

  $(b)$ for all $x, y \in V(G)$, if $\chi(x, y) \succ 0$ then $\chi(x) = \chi(y)$;

  $(c)$ every connected component of $G$ is a uniform graph.
\end{lem}
\begin{proof}
$(a)$
Take any $\alpha \in L$ and let $S$ be a connected component of $G^{(\alpha)}$. Then, by the definition of $G^{(\alpha)}$, we have that $\chi(x) = \alpha$ for all $x \in S$. Let us show that $\chi(x, y)$ is constant for all $x, y \in S$ satisfying $x \ne y$. If $|S| = 1$ or $|S| = 2$ the claim is trivial. Assume that $|S| \ge 3$. Since $S$ is a connected component, it suffices to show that whenever $x, y, z \in S$ are three distinct vertices such that $\chi(x, z) \succ 0$ and $\chi(y, z) \succ 0$, then $\chi(x, z) = \chi(y, z) = \chi(x, y)$. So, let $x, y, z \in S$ be three distinct vertices satisfying $\chi(x, z) \succ 0$ and $\chi(y, z) \succ 0$. Since $\chi(y) = \chi(z) = \alpha$, Lemma~\ref{lem.1} yields that $\chi(x, y) = \chi(x, z)$. Analogously, $\chi(x, y) = \chi(y, z)$.

$(b)$
Assume that there exist $x_1, x_2 \in V(G)$ such that $\chi(x_1, x_2) \succ 0$ and $\chi(x_1) \ne \chi(x_2)$. Without loss of generality we can assume that $\chi(x_1) \prec \chi(x_2)$. Let us now construct inductively a sequence of mappings $f_3$, $f_4$, \ldots, and a sequence of vertices $x_3$, $x_4$, $\ldots$ with the following properties:
\begin{itemize}
\item[(1)]
the mapping $\displaystyle f_{n+1} : \begin{pmatrix}
      x_{n-1} \\
      x_{n} \end{pmatrix}$
is a monomorphism from $G[x_{n-1}]$ to $G[x_{n}]$;
\item[(2)]
$G$ is $MH$-homogeneous so there is an endomorphism $f_{n+1}^*$ of $G$ which extends~$f_{n+1}$ and we let $x_{n+1} = f_{n+1}^*(x_{n})$;
\item[(3)]
$\chi(x_{i-1}) \preceq \chi(x_i)$ for all $i \in \{2, 3, \ldots, n\}$.
\end{itemize}

The mapping $\displaystyle f_3 : \begin{pmatrix} x_1 \\ x_2 \end{pmatrix}$ is easily seen to be a monomorphism from $G[x_1]$ to $G[x_2]$ since $\chi(x_1) \prec \chi(x_2)$, while in case $n \ge 3$, the requirement $(3)$ for $i = n$ ensures that $f_{n+1}$ is a monomorphism from $G[x_{n-1}]$ to $G[x_{n}]$. This shows $(1)$, and $(3)$ for $i = n+1$ follows immediately from~$(2)$. Note, also, that
$$
    \chi(x_{i-1}, x_i) = \chi(f_i^*(x_{i-2}), f_i^*(x_{i-1})) \succeq \chi(x_{i-2}, x_{i-1}), \text{ for all } i.
$$
Therefore, we have constructed a sequence of vertices $x_1, x_2, x_3, \ldots$ such that $\chi(x_1) \preceq \chi(x_2) \preceq \chi(x_3) \preceq \ldots$ and $0 \prec \chi(x_1, x_2) \preceq \chi(x_2, x_3) \preceq \chi(x_3, x_4) \preceq \ldots$. Since $\chi(x_1) \prec \chi(x_2)$ and since $G$ is finite there exists an $n$ such that $\chi(x_{n-2}) \prec \chi(x_{n-1}) = \chi(x_n)$. Then Lemma~\ref{lem.1} yields that $\chi(x_{n-2}, x_{n-1}) = \chi(x_{n-2}, x_n) \succ 0$ since $\chi(x_{n-1}) = \chi(x_n)$. By the same lemma we also have $\chi(x_{n-1}, x_n) \prec \chi(x_{n-2}, x_{n-1})$ since $\chi(x_n) \succ \chi(x_{n-2})$ . On the other hand, $\chi(x_{n-1}, x_n) \succeq \chi(x_{n-2}, x_{n-1})$ by construction. Contradiction.

$(c)$
It follows from $(b)$ that $S$ is a connected component of $G$ if and only if $S$ is a connected component of $G^{(\alpha)}$ for some $\alpha \in L$. Therefore, every connected component of $G$ is a uniform graph.
\end{proof}

\begin{thm}
Let $G$ be a finite $L$-colored graph where $L$ is a chain with the least element~0 and the greatest element~1. Then the following are equivalent:
\begin{itemize}
\item[(1)] $G$ is $HH$-homogeneous,
\item[(2)] $G$ is $MH$-homogeneous,
\item[(3)] $G$ has the following structure:
  \begin{itemize}
  \item
    every connected component of $G$ is a uniform $L$-colored graph, and
  \item
    if $U(n_1, \alpha_1, \beta_1)$ and $U(n_2, \alpha_2, \beta_2)$ are connected
    components of $G$ such that $\alpha_1 \preceq \alpha_2$, then $n_1 \le n_2$ and $\beta_1 \preceq \beta_2$.
    Consequently, if $\alpha_1 = \alpha_2$, then $n_1 = n_2$ and $\beta_1 = \beta_2$.
  \end{itemize}
\end{itemize}
\end{thm}
\begin{proof}
$(3) \Rightarrow (1)$ is easy.

$(1) \Rightarrow (2)$ is obvious.

$(2) \Rightarrow (3)$.
Let $G$ be a finite $MH$-homogeneous $L$-colored graph. We already know from Lemma~\ref{lem.2} that every connected component of $G$ is a uniform graph. So, let $S_1$ and $S_2$ be connected components of $G$ such that $G[S_1] \cong U(n_1, \alpha_1, \beta_1)$, $G[S_2] \cong U(n_2, \alpha_2, \beta_2)$ and assume that $\alpha_1 \preceq \alpha_2$. Let $x$ be an arbitrary vertex of~$S_1$ and $y$ an arbitrary vertex of~$S_2$. Then $f : \begin{pmatrix} x \\ y \end{pmatrix}$ is a monomorphism 
from $G[x]$ to $G[y]$, since $\chi(x) = \alpha_1 \preceq \alpha_2 = \chi(y)$.  So, by the homogeneity requirement, $f$ extends to an endomorphism $f^*$ of~$G$. It is easy to see that an endomorphism maps a connected component of $G$ into another connected component of $G$, so $f^*(S_1) \subseteq S_2$, since $f^*(x) = y \in S_2$. Moreover, $f^*|_{S_1}$ is injective (assume that $x, y \in S_1$ are two distinct vertices such that $f^*(x) = f^*(y)$; then $\chi(f^*(x), f^*(y)) = 0$ because $G$ is without loops; on the other hand, $\chi(f^*(x), f^*(y)) \succeq \chi(x, y) = \beta_1 \succ 0$ by the definition of an edge-uniform $L$-colored graph -- contradiction), so $n_1 = |S_1| \le |S_2| = n_2$. Finally, if $x, y \in S_1$ are two distinct vertices, then $\beta_1 = \chi(x, y) \preceq \chi(f^*(x), f^*(y)) = \beta_2$.
\end{proof}

\section{$L$-colored graphs over diamonds}
\label{s:diamonds}

In this section we consider $L$-colored graphs where $L$ is a diamond. We first consider finite vertex-uniform $L$-colored graphs and show that in this case the classes $MH$ and $HH$ coincide. We then provide an example of an $L$-colored graph which is $MH$-homogeneous, but not $HH$-homogeneous, proving thus that in the general case the classes $MH$ and $HH$ do not coincide for $L$-colored graphs where $L$ is a diamond.
So, let $L$ be a diamond with the least element~0 and the greatest element~1.

\begin{lem}\label{prop:bigraphswithredblue}
Let $G$ be a finite $MH$-homogeneous vertex-uniform $L$-colored graph and assume that there exist $x_0, y_0 \in V(G)$ such that $\chi(x_0, y_0) = 1$. Then the following holds:
\begin{enumerate}
\item[(1)] For every vertex $x$ there is a vertex $y$ such that $\chi(x, y) = 1$.
\item[(2)] Let $x,y,z$ be distinct vertices. If $\chi(x,y) = \chi(y,z) = 1$ then $\chi(x,z) = 1$.
\item[(3)] If $x$ and $y$ belong to the same connected component of $G$ then $\chi(x, y) = 1$.
\end{enumerate}
\end{lem}
\begin{proof}
(1) Let $x$ be an arbitrary vertex. Then $f = \begin{pmatrix} x_0 \\ x \end{pmatrix}$ extends to an endomorphism $f^*$ of $G$, so $\chi(x, f^*(y_0)) = \chi(x_0, y_0) = 1$.

(2) Let $\chi(x,y) = \chi(y,z) = 1$. If $\chi(x,z) \prec 1$, Lemma~\ref{lem:pump-pump} yields that $G$ then cannot be finite. Contradiction.

(3) Let $S$ be a maximal set of vertices of $G$ such that $x \in S$ and $\chi(u, v) = 1$ for all $u, v \in S$ with $u \ne v$. Note that $|S| \ge 2$ due to~(1). Let us show that $S$ coincides with the connected component $W$ of $G$ that contains~$x$. Suppose to the contrary that this is not the case and take any $z \in W \setminus S$
such that $\chi(z, y) \succ 0$ for some $y \in S$. Without loss of generality we may assume that $y \ne x$
(because $|S| \ge 2$). Note also that $\chi(x, z) \ne 1$ and $\chi(y, z) \ne 1$. Then Lemma~\ref{lem:pump-pump} yields that $G$ is not finite. Contradiction.
\end{proof}

\begin{prop}
Let $G$ be a finite $MH$-homogeneous vertex-uniform $L$-colored graph where every vertex has color $\alpha \in L$. Assume that there exist $x_0, y_0 \in V(G)$ such that $\chi(x_0, y_0) = 1$. Then there exists a positive integer $n$ such that every connected component of $G$ is isomorphic to~$U(n, \alpha, 1)$.
\end{prop}

Next, we consider finite $MH$-homogeneous vertex-uniform $L$-colored graphs satisfying $\chi(x, y) = 1$ for no $x, y \in V(G)$.

\begin{prop}\label{prop:weakconnotweakcom}
Let $G$ be a finite connected $MH$-homogeneous vertex-uniform $L$-colored graph such that $\chi(x, y) = 1$ for no $x, y \in V(G)$. Then $G$ is complete.
\end{prop}
\begin{proof}
Assume, to the contrary, that $G$ is not complete. Then there exist $x, y \in V(G)$ such that $x \ne y$ and $\chi(x, y) = 0$. Since $G$ is connected, there exists a sequence $v_1, v_2, \ldots, v_k$ of vertices of $G$ such that $x = v_1$, $y = v_k$ and $\chi(v_i, v_{i+1}) \succ 0$ for all $i \in \{1, \ldots, k-1\}$. Without loss of generality, we can assume that $(v_1, v_2, \ldots, v_k)$ is the shortest such sequence, so that $\chi(v_i, v_j) = 0$ whenever $j - i > 1$. Note that $k \ge 3$ beacuse $\chi(x, y) = 0$. Now, $f = \begin{pmatrix} v_1 & v_3 \\ v_1 & v_k \end{pmatrix}$ is a partial monomorphism which, by the homogeneity assumption, extends to an endomorphism $f^*$ of $G$. Let $z = f^*(v_2)$. Note that $\chi(x, z) \succ 0$ and $\chi(y, z) \succ 0$.
Therefore, $x$, $y$ and $z$ provide a configuration which, by Lemma~\ref{lem:pump-pump}, ensures that $G$ is not finite. Contradiction.
\end{proof}

If $G$ is a finite vertex-uniform $L$-colored graph which is connected and complete, all endomorphisms are automorphisms, and it is easy to see that $G$ is $HH$-homogeneous if and only if $G$ is $MH$-homogeneous if and only if $G$ is ultrahomogeneous.
On the other hand, if $G$ is a finite vertex-uniform $L$-colored graph wich is not connected and has the property that $\chi(x, y) \prec 1$ for all $x, y \in V(G)$, then by Proposition~\ref{prop:weakconnotweakcom} every connected component of $G$ is complete and all components have to be be isomorphic. So, we have the following partial classification result
which depends on the classification of all finite ultrahomogeneous edge colored graphs (and this is a long-standing open problem):

\begin{thm}\label{thm:general-case}
Let $L$ be a diamond with the least element~0 and the greatest element~1. The following are equivalent for a finite vertex-uniform $L$-colored graph $G$ where every vertex is colored by $\alpha \in L$:
\begin{itemize}
\item[(1)] $G$ is $HH$-homogeneous,
\item[(2)] $G$ is $MH$-homogeneous,
\item[(3)] $G$ is a disjoint union of $k \ge 1$ copies of $H$, where
\begin{itemize}
\item $H$ is $U(n, \alpha, 1)$ for some positive integer $n$; or
\item $H$ is an ultrahomogeneous $L$-colored graph such that $0 \prec \chi(x, y) \prec 1$ for all $x, y \in V(G)$ such that $x \ne y$, and $\chi(x) = \alpha$ for all $x \in V(G)$.
\end{itemize}
\end{itemize}
\end{thm}

However, if $L = M_2$ is the diamond on four elements $0, b, r, 1$ where $0 \prec b \prec 1$, $0 \prec r \prec 1$ and where $b$ and $r$ are incomparable ($b$ and $r$ stand for \emph{blue} and \emph{red}, respectively), we can provide the full classification as follows. For an $\alpha \in M_2$ let $G_{(\alpha)} = (V, E_\alpha)$ be the (ordinary undirected) graph where $E_\alpha = \{ \{x, y\} : \chi(x, y) = \alpha \}$.

\begin{thm}\label{thm:bicolored-case}
The following are equivalent for a finite vertex-uniform $M_2$-colored graph $G$ where every vertex is colored by $\alpha \in M_2$:
\begin{itemize}
\item[(1)] $G$ is $HH$-homogeneous,
\item[(2)] $G$ is $MH$-homogeneous,
\item[(3)] $G$ is a disjoint union of $k \ge 1$ copies of $H$, where
\begin{itemize}
\item $H$ is $U(n, \alpha, 1)$ for some positive integer $n$; or
\item $H$ is vertex uniform, $H_{(r)}$ is one of the Gardiner graphs and $H_{(b)}$ is its complement.
\end{itemize}
\end{itemize}
\end{thm}

As the example below shows, Theorems~\ref{thm:general-case} and~\ref{thm:bicolored-case} cannot be extended to finite $L$-colored graphs where $L$ is a diamond and graphs are not required to be vertex-uniform.

\paragraph{Example 1}
Let $G$ be an $M_2$-colored graph on four vertices $a, b, c, d$ where the vertices and the edges are colored as follows: $\chi(a) = \chi(b) = r$, $\chi(c) = \chi(d) = b$, $\chi(a, c) = \chi(c, d) = \chi(b, d) = r$, $\chi(a, d) = \chi(b, c) = b$ and $\chi(a, b) = 0$ (see Figure~\ref{fig:hh_neq_mh_example1}). 

\begin{figure}[ht]
\begin{center}
\includegraphics[scale=1.0]{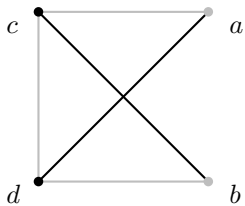}
\caption{An example of a finite $L$-colored graph that is $MH$-homogeneous but not $HH$-homogeneous.}
\label{fig:hh_neq_mh_example1}
\end{center}
\end{figure}

Then $G$ is clearly an $MH$-homogeneous graph. To see that $G$ is not an $HH$-homogeneous graph it suffices to note that the partial homomorphism $f = \begin{pmatrix} a & b \\ a & a \end{pmatrix}$ cannot be extended to an endomorphism of~$G$.

\section{Concluding remarks}
\label{s:mor-ext-classes}

A simple relational structure presented in Example 1 can easily be generalised to provide a whole class of structures that are all $MH$-homogeneous but not $HH$-homogeneous.

\begin{figure}[ht]
\begin{center}
\includegraphics[scale=1.0]{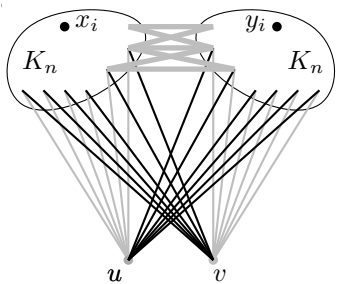}
\caption{A class of finite $L$-colored graphs which are all $MH$- but not $HH$-homogeneous.}
\label{fig:hh_neq_mh_example2}
\end{center}
\end{figure}

The construction is depicted in Figure~\ref{fig:hh_neq_mh_example2}. Fix $n \in \{1,2, \ldots , \omega\}$. (Note that in case $n = \omega$ we get an example of a countably infinite structure that is $MH$-homogeneous but not $HH$-homogeneous.)  Take two cliques both of size $n$ whose vertices and edges are colored black. Join the vertices of these two cliques by fat gray edges. Finally, add two new nonadjacent vertices $u$ and $v$ colored gray, and join the two vertices and the vertices of the two cliques by black and gray edges as in Figure~\ref{fig:hh_neq_mh_example2}. Then, as in Example 1, we can show that this graph is $MH$-homogeneous but not $HH$-homogeneous.

\begin{figure}[h!]
\begin{center}
\includegraphics[scale=1.0]{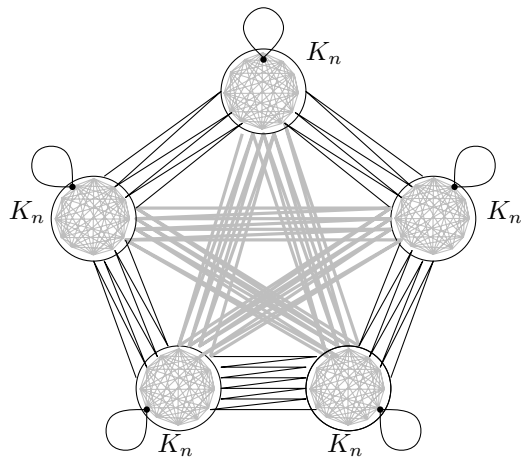}
\caption{An edge-colored graph with loops that is $MH$- but not $HH$-homogeneous.}
\label{fig:hh_neq_mh_example3}
\end{center}
\end{figure}

A question that arises immediately is whether one can avoid the need for colored vertices at the expense of introducing loops.

Consider the finite edge-colored graph depicted in Figure~\ref{fig:hh_neq_mh_example3} with no colors assigned to vertices that we construct as follows. Given $n>1$, take five copies of $K_n$ and color their edges gray. Now join these cliques by complete bipartite graphs using two mutually disjoint 5-cycles where the edges of one 5-cycle are black, while the edges of the other 5-cycle are gray. Furthermore, add a black loop to each vertex.

This graph is easily seen to be $MH$-homogeneous. To see that it is not $HH$-homogeneous, consider a partial homomorphism unifying two neighboring cliques (this is possible due to black-colored loops). Then every endomorphism that extends such a partial homomorphism would enforce the existence of an edge colored both black and gray.

\begin{figure}[ht]
\begin{center}
\includegraphics[scale=1.0]{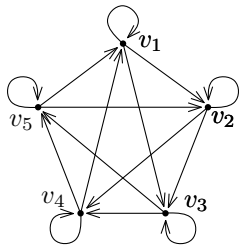}
\caption{A digraph with loops that is $MH$-homogeneous but not $HH$-homogeneous.}
\label{fig:hh_neq_mh_example4}
\end{center}
\end{figure}

Finally in Figure~\ref{fig:hh_neq_mh_example4} we present a directed graph with loops that is $MH$-homogeneous but not $HH$-homogeneous. To see that this digraph is not $HH$-homogeneous consider a partial homomorphism $f = \begin{pmatrix}
  v_1 & v_4 & v_5\\
  v_1 & v_5 & v_5
\end{pmatrix}$. Then every endomorphism that extends $f$ would enforce the existence of a bidirectional edge.

We close the paper with several open problems whose solutions would be helpfull in understanding the structure of homogeneous $L$-colored graphs with respect to various types of homogeneity discussed in this paper. Let $L$ be an arbitrary partially ordered set.

\paragraph{Problem 1}
Classify all finite $HH$-homogeneous and $MH$-homogeneous $L$-colored graphs.

\paragraph{Problem 2}
Do classes $MH$ and $HH$ coincide for finite vertex-uniform $L$-colored graphs?

\paragraph{Problem 3}
Do classes $MH$ and $HH$ coincide for countable vertex-uniform $L$-colored graphs?

\section*{Acknowledgement}
This work was supported by the Czech Science Foundation project Centre of excellence---Institute for Theoretical Computer Science (CE-ITI) number P202/12/G061, and by the Grant No.\ 174019 of the Ministry of Education and Science of the Republic of Serbia.
 
\bibliographystyle{unsrt}
\bibliography{homhom-L-col-graphs}

\end{document}